\documentclass[final]{l4dc2020} 
\usepackage[utf8]{inputenc}
\usepackage{amsfonts,bm,dsfont}
\usepackage{enumerate}
\graphicspath{{Images/}}

\usepackage{blindtext}

\newcommand{\includesvg}[2][scale=1]{\includegraphics[#1]{#2.pdf}}

\usepackage[english]{babel}

\usepackage{color}

\definecolor{Set1.red}{rgb}{0.894117647058824,0.101960784313725,0.109803921568627}
\definecolor{Set1.blue}{rgb}{0.215686274509804,0.494117647058824,0.721568627450980}
\definecolor{Set1.green}{rgb}{0.301960784313725,0.686274509803922,0.290196078431373}
\definecolor{Set1.purple}{rgb}{0.596078431372549,0.305882352941177,0.639215686274510}
\definecolor{Set1.orange}{rgb}{1,0.498039215686275,0}
\definecolor{Set1.yellow}{rgb}{1,1,0.2}
\definecolor{Set1.brown}{rgb}{0.650980392156863,0.337254901960784,0.156862745098039}
\definecolor{Set1.pink}{rgb}{0.968627450980392,0.505882352941176,0.749019607843137}
\definecolor{Set1.grey}{rgb}{0.6,0.6,0.6}

\newtheorem{assumption}{Assumption}

\usepackage[labelfont=bf]{caption}
\usepackage{float}

\usepackage{tikz}
\usetikzlibrary{arrows,arrows.meta,cd}
\usepackage{pgfplots}
\usepgfplotslibrary{groupplots}
\usepgfplotslibrary{fillbetween}

\DeclareMathOperator{\diag}{diag}

\newcommand{\vect}[2]{\ensuremath{[\begin{array}{#1} #2 \end{array}]}}
\newcommand{\norm}[1]{\ensuremath{\left\| #1 \right\|}}

\DeclareMathOperator*{\argmin}{argmin}

\DeclareMathOperator*{\minimize}{minimize}

\DeclareMathOperator{\subjectto}{subject\ to}

\newcommand{\calL}{\ensuremath{\mathcal{L}}}

\newcommand{\bzero}{\ensuremath{\bm{0}}}
\newcommand{\bA}{\ensuremath{\bm{A}}}
\newcommand{\bB}{\ensuremath{\bm{B}}}

\newcommand{\bI}{\ensuremath{\bm{I}}}

\newcommand{\bM}{\ensuremath{\bm{M}}}

\newcommand{\bQ}{\ensuremath{\bm{Q}}}
\newcommand{\bR}{\ensuremath{\bm{R}}}

\newcommand{\bLambda}{\ensuremath{\bm{\Lambda}}}

\newcommand{\be}{\ensuremath{\bm{e}}}

\newcommand{\bg}{\ensuremath{\bm{g}}}

\newcommand{\bp}{\ensuremath{\bm{p}}}

\newcommand{\bs}{\ensuremath{\bm{s}}}

\newcommand{\bu}{\ensuremath{\bm{u}}}
\newcommand{\bv}{\ensuremath{\bm{v}}}

\newcommand{\bx}{\ensuremath{\bm{x}}}

\newcommand{\bz}{\ensuremath{\bm{z}}}

\newcommand{\blambda}{\ensuremath{\bm{\lambda}}}
\newcommand{\bmu}{\ensuremath{\bm{\mu}}}

\newcommand{\bbR}{\ensuremath{\mathbb{R}}}

\def\st/{\textsuperscript{st}}
\def\nd/{\textsuperscript{nd}}
\def\rd/{\textsuperscript{rd}}
\def\th/{\textsuperscript{th}}

\newcommand{\setR}{\bbR}

\newcommand{\ones}{\ensuremath{\bm{\mathds{1}}}}

\def\nnil{\nil}
\newcounter{prob}
\newenvironment{prob}[1][\nil]{%
	\def\tmp{#1}
	\equation
	\ifx\tmp\nnil
		\refstepcounter{prob}
		\tag{P\Roman{prob}}
	\else
		\tag{\tmp}
	\fi
	\aligned%
}{%
	\endaligned\endequation%
}

\newenvironment{prob*}{%
	\csname equation*\endcsname%
	\aligned%
}{%
	\endaligned%
	\csname endequation*\endcsname%
}

\title[Counterfactual Programming for Optimal Control]{Counterfactual Programming for Optimal Control}
\usepackage{times}

\author{%
\Name{Luiz {F. O. Chamon}} \Email{luizf@seas.upenn.edu}\\
\Name{Santiago Paternain} \Email{spater@seas.upenn.edu}\\
\Name{Alejandro Ribeiro} \Email{aribeiro@seas.upenn.edu}\\
\addr Electrical and Systems Engineering, University of Pennsylvania.%
}

\begin{document}

\maketitle

\begin{abstract}
In recent years, considerable work has been done to tackle the issue of designing control laws based on observations to allow unknown dynamical systems to perform pre-specified tasks. At least as important for autonomy, however, is the issue of learning which tasks can be performed in the first place. This is particularly critical in situations where multiple~(possibly conflicting) tasks and requirements are demanded from the agent, resulting in infeasible specifications. Such situations arise due to over-specification or dynamic operating conditions and are only aggravated when the dynamical system model is learned through simulations. Often, these issues are tackled using regularization and penalties tuned based on application-specific expert knowledge. Nevertheless, this solution becomes impractical for large-scale systems, unknown operating conditions, and/or in online settings where expert input would be needed during the system operation. Instead, this work enables agents to autonomously pose, tune, and solve optimal control problems by compromising between performance and specification costs. Leveraging duality theory, it puts forward a counterfactual optimization algorithm that directly determines the specification trade-off while solving the optimal control problem.
\end{abstract}

\begin{keywords}
Autonomous systems, optimal control, constrained optimization, feasibility, constraint learning.
\end{keywords}

\section{INTRODUCTION}\label{S:intro}

Autonomous systems are machines that can modify their behavior in response to unforeseen events and/or operating conditions. They are~(or are set to become) key tools in robotics~\citep{Kober13r} and smart~(grid, home, city) applications~\citep{Gharaibeh17s}. One important aspect of autonomy that has attracted considerable attention in recent years is that of learning to perform tasks in uncertain or unknown environments. Here, the agent has limited~(or no) knowledge of the underlying dynamical system and operational conditions and must design control laws to perform a pre-specified task based only on observations. System identification~\citep{Johansson93s}, adaptive control~\citep{Kokotovic91f}, and reinforcement learning~\citep{Sutton18r}, to name a few, have been used to address these issues.

Though learning \emph{how} to perform tasks is paramount to achieve autonomous behavior, at least as important is the issue of learning \emph{which} tasks can be performed in the first place. This decision-making aspect of autonomy is both fundamental and challenging, especially when agents must make decisions that \emph{violate} their specifications. This is critical when multiple tasks and constraints are simultaneously required from the agent, resulting in infeasible settings. These situations arise due to over-specification, scenario uncertainty, or changing operating conditions, and are only aggravated when dynamical system models are learned through simulations.

When faced with the problem of dealing with~(possibly conflicting) goals and requirements, agents are often designed using domain expert knowledge and/or prior information to weight and combine multiple objectives into a single cost, leading to so-called \emph{regularized problems}. Although solutions obtained using this approach are often Pareto optimal~\citep{Ehrgott05m, Boyd04c}, i.e., none is better than the others in every aspect, their performance in practice can vary widely. Thus the importance of properly tuning the regularization parameters, a task that is both application- and scenario-specific. This well-known issue has lead to several alternative formulations and heuristics~\citep{Miettinen98n, Das98n, Messac03t, Mueller-Gritschneder09a}. Alternatively, chance constraints can be used to relax the hard requirements by imposing a minimum probability of satisfying them, an approach often used in controlled Markov decision processes~\citep{Howard72r, Geibel05r, Paternain19s, Paternain19c}, optimal control~\citep{Caillau18s, Ono10c}, and model predictive control~\citep{Schwarm99c, Li00r}.

The main issue with these approaches is that they heavily rely on expert knowledge to inform the solution of multi-objective problems. Hence, adapting to the varying trade-offs of non-stationary environments would require additional input from experts during the system operation, renouncing on a key component of autonomy. What is more, designing regularization parameters or chance constraints that encode trade-off preferences is a challenging task in and of itself. Since properly encoding this information may be difficult, it may lead to solutions with poor practical performance despite their theoretical optimality.

This work provides a systematic approach to tune the agent specifications according to the inherent trade-offs of the underlying dynamical system and operating environment. This approach leverages counterfactuals, conditional statements that describe alternative versions of the world, to interrogate optimal control problems as to what would happen if the specifications had been different. This counterfactual evidence can then be used to automatically tune compromises and update requirements in a principled way without having to repeatedly solve different versions of the control problem. To do so, it first puts forward a mathematical formulation of compromise~(Section~\ref{S:problem}) and shows that Lagrange multipliers can be used to counterfactually determine specifications that balance performance gains and costs~(Section~\ref{S:counterfactual}). This result leads to a low complexity saddle point algorithm that simultaneously solves the original optimization problem while tuning its compromises~(Section~\ref{S:arrow}). This method enables, for instance, an agent to autonomously pose control problems adapted to its operating conditions~(Section~\ref{S:sims_lqr}).

\section{PROBLEM FORMULATION}
	\label{S:problem}

Let~$f_0: \setR^n \to \setR$ be a performance metric we wish to optimize while abiding by a set of requirements determined by the functions~$f_i: \setR^n \to \setR$ and the specifications~$s_i \geq 0$, $i = 1,\dots,m$. These requirements may represent tasks that the agent must perform, specify system limitations and/or available resources, or describe desired behaviors~(see example in Section~\ref{S:sims_lqr}). Formally, we consider the optimization program
\begin{prob}\label{P:relaxed}
	p^\star(\bs) \triangleq \min_{\bz \in \mathbb{R}^n}& &&f_0(\bz)
	\\
	\text{s. to}& &&f_i(\bz) \leq s_i
		\text{,} \quad i = 1,\dots,m
		\text{,}
\end{prob}
where the vector~$\bs \in \setR_+^m$ collects the specifications~$s_i$. Problem~\eqref{P:relaxed} explicitly captures the trade-off between performance and requirements. Indeed, let~$\bs = \bzero$ denote a nominal, reference specification and denote its performance by~$p^\star = p^\star(\bzero)$. Then, if~$s_i > 0$, the~$i$-th requirement becomes easier to satisfy and the performance~$p^\star(\bs)$ of the solution improves accordingly. Naturally, there is a cost associated with modifying requirements that must be taken into account when adjusting~$\bs$. To do so, let~$h: \setR_+^m \to \setR_+$ be a non-decreasing function in each argument, so that~$h(\bs)$ describes, in units of performance~($f_0$), the cost associated with specification~$\bs$. Without loss of generality, we assume the nominal specification to be free, i.e., $h(\bzero) = 0$.

Our goal is to tune~$\bs$ so as to trade off the specification cost~$h(\bs)$ and the performance improvement~$p^\star - p^\star(\bs)$, i.e., to find specifications that cost at most as much as they enhance performance. Formally, we seek~$\bs^\dagger$ such that
\begin{equation}\label{E:special_compromise}
	p^\star - p^\star(\bs^\dagger) \geq h(\bs^\dagger)
		\text{.}
\end{equation}
Though useful, \eqref{E:special_compromise} does not consider the case in which the nominal requirements are either conflicting or so stringent that~\eqref{P:relaxed} is infeasible, i.e., there is no~$\bz^\prime \in \setR^n$ such that~$f_i(\bz^\prime) \leq 0$ for all~$i$. In this case, we define~$p^\star(\bzero) = +\infty$. Under these circumstances, we wish to fall back into a laxer notion of feasibility by choosing a different reference specification in~\eqref{E:special_compromise}. Since this choice is arbitrary, we require instead that the compromise~$\bs^\dagger$ holds for all valid references, i.e.,
\begin{equation}\label{E:compromise}
	p^\star(\bs_0) - p^\star(\bs^\dagger) \geq h(\bs^\dagger) - h(\bs_0)
		\text{,}
\end{equation}
for all~$\bs_0 \in \setR^m_+$ such that~\eqref{P:relaxed} is feasible, i.e., such that~$p^\star(\bs_0) < +\infty$. We call~$\bs^\dagger$ a compromise because it describes a specification of~\eqref{P:relaxed} that trades off performance improvement and cost in a manner that is more profitable that any other feasible specification. Observe that the compromise in~\eqref{E:compromise} takes into account not only the specification cost, but also how hard the requirement is to satisfy in the first place. To see this is the case, note that the left-hand side of~\eqref{E:special_compromise} is a measure of constraint satisfaction difficulty for the nominal specification of~\eqref{P:relaxed}. Indeed, for a fixed specification budget~$\delta > 0$, the value~$\Delta_i = p^\star - p^\star(\delta \be_i)$, where~$\be_i$ is the~$i$-th vector of the canonical basis, quantifies the performance enhancement obtained by relaxing the~$i$-th requirement.

It is not immediate from~\eqref{P:relaxed} that the equilibrium~\eqref{E:compromise} exists and, if it does, whether it can be determined efficiently. Indeed, note that~\eqref{E:compromise} involves the optimal value of~\eqref{P:relaxed}. If finding~$\bs^\dagger$ involves repeatedly solving~\eqref{P:relaxed}, the computational costs would hinder the usefulness of this approach in practice, specially in online settings~[e.g., for MPC~\citep{Borrelli17p, Rawlings17m}]. In the next section, we show that under mild conditions, the compromise~$\bs^\dagger$ in~\eqref{E:compromise} exists and can be determined counterfactually, i.e., without repeatedly solving~\eqref{P:relaxed}. Based on these results, we then put forward a modified Arrow-Hurwicz algorithm that directly solves~\eqref{P:relaxed} for~$\bs^\dagger$ without checking multiple specifications~(Section~\ref{S:arrow}).

Before proceeding, we introduce a pertinent remark addressing a common alternative approach for dealing with the trade-off of multiple objectives.

\begin{remark}
A typical approach for dealing with multi-objective optimization problems, such as requirement-performance trade-offs, and infeasibility is to use regularization, wherein~\eqref{P:relaxed} is replaced by~$\min_{\bz \in \setR^n}\ f_0(\bz) + \sum_{i=1}^m \gamma_i f_i(\bz)$ for regularization parameters~$\gamma_i \geq 0$~\citep{Miettinen98n, Ehrgott05m}. However, whereas~\eqref{E:compromise} explicitly states the trade-off of interest, this compromise is implicit in the regularized version through the relation among the~$\gamma_i$ and between the~$\gamma_i$ and the performance objective~$f_0$. Since there is typically no straightforward way to tune~$\gamma_i$, they are typically selected by domain experts by trial-and-error or computationally intensive grid searches~(e.g., cross-validation). In contrast, Algorithm~\ref{A:saddle} can be used to efficiently determine~$\bs^\dagger$.
\end{remark}

\section{LAGRANGE MULTIPLIERS AS COUNTERFACTUALS}
\label{S:counterfactual}

In the end of Section~\ref{S:problem}, we argued that finding~$\bs^\dagger$ in~\eqref{E:compromise} is only feasible in practice if it does not involve repeatedly solving~\eqref{P:relaxed}. Our goal in this section is to extract counterfactual evidence from~\eqref{P:relaxed} to tune its requirements so as to achieve~\eqref{E:compromise} without testing multiple specifications. A \emph{counterfactual} is a conditional proposition in which the premise is false and the consequent describes how the world would have been \emph{if the premise were true}~\citep{Pearl09c, Woodward05m, Lewis74c}. It is immediate that~\eqref{P:relaxed} yields counterfactuals of the form ``if the requirements had been~$\bs$, then the optimal performance value would have been~$p^\star(\bs)$.'' This causal relation is represented by the blue arrow in Fig.~\ref{F:counterfactual}. Less straightforward is the fact that~\eqref{P:relaxed} can also provide counterfactual evidence for the trade-off~\eqref{E:compromise}~(red arrows in Fig.~\ref{F:counterfactual}). Next, we show that this evidence comes from the dual variables of~\eqref{P:relaxed}, which allows us to derive an algorithm that directly solves~\eqref{P:relaxed} for~$\bs^\dagger$.

To proceed, start by defining the Lagrangian associated with~\eqref{P:relaxed} as
\begin{equation}\label{eqn_lagrangian}
	\calL(\bx,\blambda,\bs) = f_0(\bx) + \sum_{i = 1}^m \lambda_i f_i(\bx)
		- \blambda^T \bs
		\text{,}
\end{equation}
where~$\blambda \in \setR_+^m$ collects the dual variables~$\lambda_i \geq 0$ for~$i = 1,\dots,m$. Likewise, define its dual function as~$g(\blambda,\bs) = \min_{\bx\in\setR^n} \calL(\bx,\blambda,\bs)$. The dual function is a lower bound on~$p^\star(\bs)$ for all~$\blambda,\bs \in \setR_+^m$. The dual value is the best of these lower bounds, namely~$d^\star(\bs) \triangleq \max_{\blambda \in \setR_+^m} g(\blambda,\bs)$. Under Assumptions~\ref{assumption_convexity} and~\ref{assumption_slater} below, this best lower bound can be shown to attain~$p^\star(\bs)$, i.e., $d^\star(\bs) = p^\star(\bs)$. In duality theory, this is known as \emph{strong duality}~\citep[Ch.~5]{Boyd04c}.

\begin{assumption}\label{assumption_convexity}
The~$f_i$ and~$h$ are differentiable and convex and~$f_0$ is differentiable and strongly convex.
\end{assumption}

\begin{assumption}\label{assumption_slater}
The set $\left\{(\bx^\prime,\bs^\prime) \in \setR^n \times \setR_+^m \mid f_i(\bx^\prime) < \bs^\prime \text{ for } i = 1,\dots,m \right\}$ is not empty.
\end{assumption}

Under these conditions, it is well-known that the optimal dual variable~$\lambda_i^\star(\bs)$ of~\eqref{P:relaxed} identifies how hard the $i$-th constraint is to satisfy. Specifically, they locally quantify how much the objective would change if the constraint were tightened or relaxed. The following theorem provides an additional counterfactual property by showing that it also uniquely identifies the compromise~\eqref{E:compromise}.

\begin{theorem}\label{T:counterfactual}
Let~$\blambda^\star(\bs)$ be the dual variables of problem~\eqref{P:relaxed} with slack~$\bs$ and~$\bs^\dagger$ be the compromise in~\eqref{E:compromise}. Under Assumptions~\ref{assumption_convexity} and~\ref{assumption_slater},
\begin{equation}\label{E:counterfactual}
	\nabla h(\bs) = \blambda^\star(\bs) \Leftrightarrow \bs = \bs^\dagger
		\text{.}
\end{equation}
\end{theorem}

\begin{proof}
We start by proving necessity~($\Rightarrow$). Note that under Assumptions~\ref{assumption_convexity} and~\ref{assumption_slater}, $h$ and~$p^\star$ are convex, differentiable functions. Indeed, $p^\star$ is the \emph{perturbation function} of~\eqref{P:relaxed}, a strongly dual, convex program with a strongly convex objective~\citep{Shapiro00d, Bertsekas09c}. Hence, for all~$\bs_0,\bs \in \setR_+^m$ such that~$p^\star(\bs_0) < +\infty$ and~$p^\star(\bs) < +\infty$ it holds that
\begin{equation}\label{E:convex_inequalities}
	p^\star(\bs_0) \geq p^\star(\bs) + \nabla p^\star(\bs)^T (\bs_0 - \bs)
	\quad \text{and} \quad
	h(\bs_0) \geq h(\bs) + \nabla h(\bs)^T (\bs_0 - \bs)
		\text{,}
\end{equation}
which add up to
\begin{equation}\label{E:summed_inequalities}
\begin{aligned}
	p^\star(\bs_0) - p^\star(\bs) &\geq h(\bs) - h(\bs_0)
		+ \left[ \nabla p^\star(\bs) + \nabla h(\bs) \right]^T (\bs_0 - \bs)
		\text{.}
\end{aligned}
\end{equation}
Using the fact that~$\blambda^\star(\bs) = -\nabla p^\star(\bs)$ for all~$\bs \in \setR_+^m$ so that~\eqref{P:relaxed} is feasible~\citep[Section 5.6.3]{Boyd04c}, the hypothesis from~\eqref{E:counterfactual} yields~$\nabla p^\star(\bs) + \nabla h(\bs) = \bzero$. Hence, \eqref{E:summed_inequalities} imply~\eqref{E:compromise} and we obtain that~$\bs = \bs^\dagger$.

Sufficiency in~\eqref{E:counterfactual}~($\Leftarrow$) stems from the fact that~\eqref{E:compromise} can be rearranged into~$p^\star(\bs^\dagger) + h(\bs^\dagger) \leq p^\star(\bs_0) + h(\bs_0)$. Hence, $\bs^\dagger$ must be the unique minimizer of the strongly convex function~$q(\bs) = p^\star(\bs) + h(\bs)$, so that~$\nabla p^\star(\bs^\dagger) + \nabla h(\bs^\dagger) = \bzero$. Once again using the fact that~$\blambda^\star(\bs) = -\nabla p^\star(\bs)$ for all~$\bs$~\cite[Section 5.6.3]{Boyd04c} concludes the proof.
\end{proof}

Theorem~\ref{T:counterfactual} replaces the global compromise between~$p^\star$ and~$h$ in~\eqref{E:compromise} by the local relation between~$\blambda^\star(\bs)$ and~$h(\bs)$ in~\eqref{E:counterfactual}~(red arrows in Fig.~\ref{F:counterfactual}). Effectively, it states that the dual variables~$\blambda^\star(\bs)$ of~\eqref{P:relaxed} can be used to determine~$\bs^\dagger$ without solving it for different specifications. Indeed, \eqref{E:counterfactual} provides counterfactual of the form ``if $\nabla h(\bs)$  had been a dual variable of~\eqref{P:relaxed}, then~$p^\star(\bs)$ and~$h(\bs)$ would have obeyed~\eqref{E:compromise}.'' For~\cite{Pearl09c}, $\nabla h(\bs) = \blambda^\star(\bs)$ is the ``surgical intervention'' used to modify the causal model in Fig.~\ref{F:counterfactual} to evaluate this counterfactual. We could then state that~$\nabla h(\bs) = \blambda^\star(\bs)$ causes~\eqref{E:compromise}. Theorem~\ref{T:counterfactual} also provides a backtracking counterfactual of the form ``if~$p^\star(\bs)$ and $h(\bs)$ were to obey~\eqref{E:compromise}, then~$\nabla h(\bs) = \blambda^\star(\bs)$.'' By enforcing~\eqref{E:counterfactual} as we solve~\eqref{P:relaxed}, we therefore simultaneously obtain~$\bs^\dagger$ and~$p^\star(\bs^\dagger)$, i.e., solve~\eqref{P:relaxed} for the compromise~\eqref{E:compromise}. In the sequel, we derive a method to do so based on a modified Arrow-Hurwicz algorithm~\citep{Arrow58s}.

\subsection{A modified Arrow-Hurwicz algorithm}
\label{S:arrow}

Theorem~\ref{T:counterfactual} suggests a way to exploit the counterfactual information in the dual variables~$\blambda$ to directly obtain a solution of~\eqref{P:relaxed} for the optimal slack~$\bs^\star$ without ever solving~\eqref{P:relaxed}. Indeed, observe that due to Assumptions~\ref{assumption_convexity} and~\ref{assumption_slater}, it holds that the primal-dual solution~$(\bx^\star(\bs),\blambda^\star(\bs))$ is a saddle point of the Lagrangian~\eqref{eqn_lagrangian}~\citep[Section 5.4.2]{Boyd04c}, i.e.,
\begin{equation}\label{eqn_saddle_point_1}
	\calL(\bx^\star(\bs),\blambda,\bs) \leq \calL(\bx^\star(\bs),\blambda^\star(\bs),\bs)
		\leq \calL(\bx,\blambda^\star(\bs),\bs)
\end{equation}
for all~$\bx \in \setR^n$, $\blambda \in \setR_+^m$, and~$\bs$ such that Assumption~\ref{assumption_slater} holds. In the sequel, we put forward a procedure to find points that satisfy both~\eqref{E:counterfactual} and~\eqref{eqn_saddle_point_1}.

Start by considering the classic Arrow-Hurwicz algorithm for solving~\eqref{P:relaxed} when the slacks~$\bs$ are constant~\citep{Arrow58s}. This method seeks a saddle point as in~\eqref{eqn_saddle_point_1} by updating the primal and dual variables using gradients of the Lagrangian~\eqref{eqn_lagrangian}. Explicitly, the primal variables~$\bx$ are updated by descending along the negative gradient of the Lagrangian
\begin{equation}\label{eqn_gradient_descent}
	\dot{\bx} = -\nabla_{\bx} \calL(\bx,\blambda,\bs) = 
		- \left( \nabla f_0(\bx)+ \sum_{i=1}^m \lambda_i \nabla f_i(\bx) \right)
		\text{,}
  \end{equation}
and the dual variables~$\blambda$ are updated by ascending along the gradient of the Lagrangian as in
\begin{equation}\label{eqn_gradient_ascent}
	\dot{\blambda} =
		\Pi_{\mathbb{R}^m_+} \left[ \blambda,\nabla_{\blambda} \calL(\bx,\blambda,\bs) \right]
		= \Pi_{\mathbb{R}^m_+}\left[ \blambda,f_i(\bx) - s_i \right],
\end{equation}
where~$\Pi_{\setR_+^m}$ refers to a projected dynamical system over the positive orthant of~$\setR^m$~\citep{Nagurney96p}. This projection is introduced to ensure that the Lagrange multipliers are non-negative.

\begin{algorithm}[t]
\begin{minipage}[b]{0.55\columnwidth}
\begin{algorithm}[H]
	Let $\bx^{(0)} = \bzero$, $\blambda^{(0)} = \ones$, $\bs^{(0)} = \ones$,
		and~$0 < \eta \ll 1$.

	\For{$t = 1,2,\dots$}{
	
		$\displaystyle
			\bg_x^{(t)} = \nabla f_0\left( \bx^{(t-1)} \right)
				+ \sum_{i=1}^m \lambda_i^{(t-1)} \nabla f_i\left( \bx^{(t-1)} \right)$
		\\
		$\displaystyle
			\bx^{(t)} = \bx^{(t-1)} - \eta \bg_x^{(t)}$
		\\[6pt]
		$\displaystyle
			\lambda_i^{(t)} = \left[ \lambda_i^{(t-1)}
				+ \eta \left( f_i\left( \bx^{(t-1)} \right) - s_i^{(t-1)} \right) \right]_+$
	}
\caption{Counterfactual optimization algorithm}
	\label{A:saddle}
\end{algorithm}
\end{minipage}
\hfill
\begin{minipage}[b]{0.42\columnwidth}
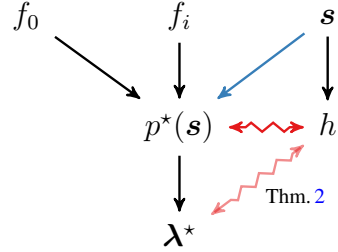
\begin{figure}[H]
\centering
\large
\tikzcdset{arrow style=tikz, diagrams={>=stealth'}}
\begin{tikzcd}[ampersand replacement=\&, arrows={line width=0.9pt}]
	f_0 \ar[rd] \& f_i \ar[d]
		\& \bs \ar[ld, Set1.blue] \ar[d]
	\\
	\& p^\star(\bs) \ar[d]
	\& h \ar[l, leftrightsquigarrow, Set1.red]
		\ar[leftrightsquigarrow, Set1.red, opacity=0.5]{dl}%
			[black, opacity=1]{\textup{Thm.\,\ref{T:counterfactual}}}
	\\
	\& \blambda^\star \& 
\end{tikzcd}
\vspace{12pt}
\caption{Causal diagram}
	\label{F:counterfactual}
\end{figure}
\end{minipage}
\end{algorithm}

To understand the intuition behind this algorithm, observe that the primal variable is updated in~\eqref{eqn_gradient_descent} by descending along a weighted combination of gradients from the objective and the constraints so as to reduce the value of all functions. The value of the weight of each constraint is given by its respective dual variable~$\lambda_i$. If constraint~$i$ is satisfied, its Lagrangian multiplier is zero, so that its influence on the primal update is decreased by~\eqref{eqn_gradient_ascent}. On the other hand, if this constraint is violated, then~$f_i(\bx) - s_i > 0$ and the value of the corresponding multiplier is increased. The relative strength of each gradient in the primal update~\eqref{eqn_gradient_descent} is therefore related to the history of violation of each constraint.

The main drawback of the Arrow-Hurwicz dynamics as they stand is that they solve~\eqref{P:relaxed} for a fixed slack~$\bs$. However, the compromise~$\bs^\dagger$ in~\eqref{E:compromise} is unknown \emph{a priori}. To overcome this limitation, we use~\eqref{E:counterfactual} to replace~\eqref{eqn_gradient_ascent} by
\begin{equation}\label{eqn_slacks}
	\dot{\blambda}
		= \Pi_{\mathbb{R}^m_+}\left[ \blambda,f_i(\bx) - (\nabla h)^{-1}(\blambda) \right]
		\text{.}
\end{equation}
The inverse of~$\nabla h$ exists since~$h$ is strongly convex~(Assumption~\ref{assumption_convexity}). Note that~\eqref{eqn_slacks} takes~$\bs = (\nabla h)^{-1}(\blambda)$, i.e., it enforces that the specifications of~\eqref{P:relaxed} satisfy~\eqref{E:counterfactual}. Hence, \eqref{eqn_gradient_descent}--\eqref{eqn_slacks} update the primal/dual variables and specifications~$\bs$ such as to solve~\eqref{P:relaxed} directly for~$\bs^\dagger$ in~\eqref{E:compromise}. A discretized version of the counterfactual optimization method is shown in Algorithm~\ref{A:saddle}.

The dynamics~\eqref{eqn_gradient_descent}--\eqref{eqn_slacks} can be shown to converge to a point that satisfies the saddle point relation~\eqref{eqn_saddle_point_1} as well as the left-hand side of~\eqref{E:counterfactual} using an argument similar to~\citep{Cherukuri16a}. From Theorem~\ref{T:counterfactual}, they therefore simultaneously obtain the specification~$\bs^\dagger$ that satisfies~\eqref{E:compromise} and the solution~$\bz^\star(\bs^\star)$ of~\eqref{P:relaxed} that achieve~$p^\star(\bs^\dagger)$. Due to space constraints, details of this proof are left for a future version of this manuscript.

\section{SPECIFYING CONTROLLERS FOR UNKNOWN DYNAMICAL SYSTEMS}
\label{S:sims_lqr}

Control systems must often trade off control input cost~(or energy) and regulation~\citep{Anderson07o, Bertsekas17d}. The compromise between these objectives depends on the underlying dynamical system, operating conditions, and goals of the agent. What is more, this compromise may need to be revised to accommodate changes in the operating conditions, due to non-stationary environment and model uncertainty. Typically, this trade off is adjusted using domain expert knowledge of the problem. Autonomous systems, however, must be able to automatically tune these specifications without human intervention, a feature especially critical in online applications~[e.g., MPC~\citep{Borrelli17p, Rawlings17m}].

\begin{figure}[tb]

\begin{minipage}[c]{0.48\columnwidth}
\centering
\includesvg[width=\columnwidth]{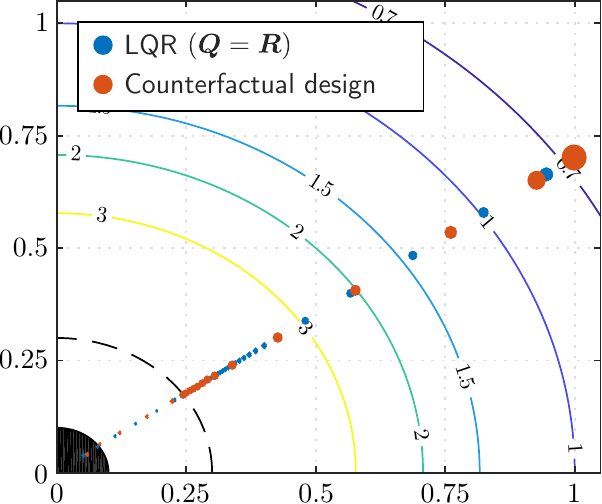}
\end{minipage}
\hfill
\begin{minipage}[c]{0.5\columnwidth}
\centering
\includesvg[width=0.98\columnwidth]{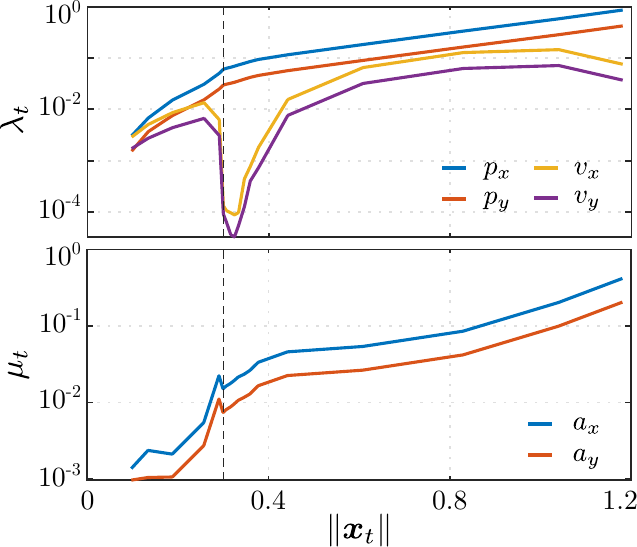}
\end{minipage}

\caption{(a)~Trajectory of agents in an environment with position-dependent friction~(markers size are proportional to the input energy used at each step): LQR~\eqref{P:lqr}~(blue curve) and counterfactual controller~\eqref{P:cf_lqr}~(yellow curve). (b)~Counterfactual specifications of~\eqref{P:cf_lqr}.}
\label{F:friction_iso}

\end{figure}

This problem can be tackled using counterfactual optimization by formulating it in the language of compromise from Section~\ref{S:problem}. To do so, we note that the ideal, yet infeasible, controller specification is one that regulates the system in one step without acting on it. Explicitly, it is the nominal specification of
\begin{prob}\label{P:cf_lqr}
	\text{find}& &&\bu_1, \dots, \bu_T
	\\
	\text{such that}& &&\left( [\bx_t]_i \right)^2 \leq s_{x,i}
		\text{,} &i &= 1,\dots,\ell
	\\
	&&&\left( [\bu_t]_j \right)^2 \leq s_{u,j}
		\text{,} &j &= 1,\dots,p
	\\
	&&&\bx_t = \bA \bx_{t-1} + \bB \bu_t
		\text{,} &t &= 1,\dots,T
\end{prob}
where~$\bx_t \in \setR^\ell$ is the state vector, $\bu_t \in \setR^p$ are the control actions, and~$g_t$ describes the dynamics of the system at time~$t$. We write~$[\bx]_i$ to denote the~$i$-th entry of the vector~$\bx$. The initial state~$\bx_0 \in \setR^n$ is assumed to be given. When~$g_t$ describes linear dynamics, this problem is convex and can be written as in~\eqref{P:relaxed} by taking~$f_0(\bx) = 0$. Notice that the nominal specification of~\eqref{P:cf_lqr} for which~$\bs_x = \bs_u = \bzero$ is such that~$\norm{\bx_t} = \norm{\bu_t} = 0$. Thus, unless~$\bx_0 = \bzero$, \eqref{P:cf_lqr} is infeasible for these specifications. Algorithm~\ref{A:saddle} can then be used to tune~$(\bs_x,\bs_u)$ to trade-off regulation and input energy specifications. In the sequel, we illustrate this procedure in the context of navigating an unknown terrain.

Consider an agent navigating an unknown terrain, modeled as a position-dependent friction coefficient. The states of the underlying dynamical system are~$\bx = \vect{cccc}{\bp^T & \bv^T}^T$, describing the position~$\bp = \vect{cc}{p_x & p_y}^T$ and velocities~$\bv = \vect{cc}{v_x & v_y}^T$ of the agent, and its control inputs are accelerations collected in~$\bu = \vect{cc}{a_x & a_y}^T$. Its state-space representation is~$\dot{\bx} = \bA \bx + \bB \bu$ for
\begin{equation}\label{E:stateSpaceCart}
	\bA = \begin{bmatrix}
	0 & 0 & 1 & 0
	\\
	0 & 0 & 0 & 1
	\\
	0 & 0 & -\gamma(\bp) & 0
	\\
	0 & 0 & 0 & -\gamma(\bp)
	\end{bmatrix}
	\text{, }
	\bB = \begin{bmatrix}
	0 & 0 
	\\
	0 & 0 
	\\
	1 & 0 
	\\
	0 & 1
	\end{bmatrix}
	\text{, and }
	\gamma(\bp) = \begin{cases}
		\norm{\bp}^{-2} \text{,} &\norm{\bp} > 0.3
		\\
		0 \text{,} &\norm{\bp} \leq 0.3
	\end{cases}
	\text{,}
\end{equation}
where~$\gamma(\bp)$ is the gyroscopic friction coefficient at coordinates~$\bp$. In other words, the terrain becomes harder to navigate as the agent approaches the origin~(the goal), until we reach a slippery region in which there is no friction. We assume that~$\gamma$ can be measured by the agent at its current position only, i.e., no prior information is available about its value in the environment. In the sequel, we consider a discretized version of~\eqref{E:stateSpaceCart} with sampling time~$T_s = 0.5$~s.

To deal with changes in the environment, an MPC controller is used by planning for a horizon~$T = 3$~iterations, but applying only the first action. Since the agent only has access to local information on the environment, the actions are planned assuming that the state transition matrix~$\bA$ is constant and that the gyroscopic friction is that of the current position. Hence, \eqref{P:cf_lqr} is a strongly convex program. For comparison, Figure~\ref{F:friction_iso}a~(blue curve) displays the trajectory obtained for the classical LQR
\begin{prob}\label{P:lqr}
	\minimize& &&\sum_{t = 1}^T \left( \bx_t^T \bQ_t \bx_t
		+ \bu_t^T \bR_t \bu_t \right)
	\\
	\subjectto& &&\bx_t = \bA \bx_{t-1} + \bB \bu_t
		\text{,} \quad t = 1,\dots,T
		\text{,}
\end{prob}
with~$\bQ_t = \bI$ and~$\bR_t = \bI$. Notice that as the agent approaches the origin while outside the slippery region, it begins to move slowly, i.e., taking small steps. This occurs because the relative importance between regulation and input energy is set independently of the local friction coefficient in~\eqref{P:lqr}. The agent is therefore unwilling to spend the extra energy needed to overcome the high friction region faster. Thus, it takes~$53$~iterations~(approximately~$26$~s) to reach~$\norm{\bx_t} \leq 0.1$. Though in certain applications this is the desired behavior, i.e., there is no room for trading off input energy and regulation, it is clear that agents operating in dynamic conditions can benefit from being allowed to autonomously tune their specifications. Indeed, the trajectory obtained by counterfactually solving~\eqref{P:cf_lqr} using~$h(\bs) = \norm{\bs}^2$~(Figure~\ref{F:friction_iso}a, yellow curve) attains~$\norm{\bx_t} \leq 0.1$ in~$19$~iterations spending~$\sum_t \norm{\bu_t}^2 = 1.5$. By modifying the relative costs in~\eqref{P:lqr} using~$\bQ_t = \gamma \bI$, the same completion time can be achieved using~$\sum_t \norm{\bu_t}^2 = 1.81$~($\gamma = 3.3$), showing the price of using a fixed trade-off in a dynamic scenario.

Interestingly, it turns out that~$\bQ_t$ and~$\bR_t$ can be tuned so as to achieve the same performance as the controller obtained from~\eqref{P:cf_lqr} due to the fact that~\eqref{P:cf_lqr} is related to~\eqref{P:lqr} by Lagrangian duality. Indeed, the Lagrangian of~\eqref{P:cf_lqr} is given by
\begin{equation}\label{E:cf_lqr_equivalence}
\begin{aligned}
	\calL(\bx_t,\bu_t,\blambda_t,\bmu_t) &= \sum_{t = 1}^T \Bigg[
		\sum_{i = 1}^\ell \lambda_{t,i} \left( \left( [\bx_t]_i \right)^2 - s_{x,i} \right)
	+ \sum_{j = 1}^p \mu_{t,j} \left( \left( [\bu_t]_j \right)^2 - s_{u,j} \right)
	\Bigg]
	\\
	{}&= \sum_{t = 1}^T \left( \bx_t^T \bLambda_t \bx_t
		+ \bu_t^T \bM_{t} \bu_t^T
		- \blambda_t^T \bs_{x} - \bmu_t^T \bs_{u}
	\right)
		\text{,}
\end{aligned}
\end{equation}
where~$\bLambda_t = \diag(\lambda_{t,i})$ and~$\bM_t = \diag(\mu_{t,j})$. It is straightforward from~\eqref{E:cf_lqr_equivalence} that if~$(\bx_t^\star,\bu_t^\star)$ is a solution of~\eqref{P:lqr} with~$\bQ_t = \bLambda_t$ and~$\bR_t = \bM_t$, then~$(\bx_t^\star,\bu_t^\star) = \argmin_{(\bx_t,\bu_t)} \calL(\bx_t,\bu_t,\blambda_t,\bmu_t)$ for the Lagrangian in~\eqref{E:cf_lqr_equivalence}. Using an appropriate sequence of~$\bQ_t$ and~$\bR_t$ in~\eqref{P:lqr} therefore yields the same sequence of control actions as counterfactually solving~\eqref{P:cf_lqr}. Nevertheless, even with complete knowledge of the environment, manually selecting this sequence is a challenging task as the weights do not have a straightforward behavior~(Figure~\ref{F:friction_iso}b).

\bibliography{IEEEabrv,af,bayes,control,gsp,math,ml,rkhs,rl,sp,stat}

\begin{thebibliography}{31}
\providecommand{\natexlab}[1]{#1}
\providecommand{\url}[1]{\texttt{#1}}
\expandafter\ifx\csname urlstyle\endcsname\relax
  \providecommand{\doi}[1]{doi: #1}\else
  \providecommand{\doi}{doi: \begingroup \urlstyle{rm}\Url}\fi

\bibitem[Anderson and Moore(2007)]{Anderson07o}
B.D.O. Anderson and J.B. Moore.
\newblock \emph{Optimal Control: {L}inear Quadratic Methods}.
\newblock Dover, 2007.

\bibitem[Arrow et~al.(1958)Arrow, Hurwicz, and Uzawa]{Arrow58s}
K.J. Arrow, L.~Hurwicz, and H.~Uzawa.
\newblock \emph{Studies in linear and non-linear programming}.
\newblock Stanford University Press, 1958.

\bibitem[Bertsekas(2009)]{Bertsekas09c}
D.P. Bertsekas.
\newblock \emph{Convex Optimization Theory}.
\newblock Athena Scientific, 2009.

\bibitem[Bertsekas(2017)]{Bertsekas17d}
D.P. Bertsekas.
\newblock \emph{Dynamic Programming and Optimal Control -- {Vol.~I}}.
\newblock Athena Scientific, 2017.

\bibitem[Borrelli et~al.(2017)Borrelli, Bemporad, and Morari]{Borrelli17p}
F.~Borrelli, A.~Bemporad, and M.~Morari.
\newblock \emph{Predictive Control for Linear and Hybrid Systems}.
\newblock Cambridge University Press, 2017.

\bibitem[Boyd and Vandenberghe(2004)]{Boyd04c}
S.~Boyd and L.~Vandenberghe.
\newblock \emph{Convex optimization}.
\newblock Cambridge University Press, 2004.

\bibitem[Caillau et~al.(2018)Caillau, Cerf, Sassi, Tr{\'e}lat, and
  Zidani]{Caillau18s}
J-B Caillau, Max Cerf, Achille Sassi, Emmanuel Tr{\'e}lat, and Hasnaa Zidani.
\newblock Solving chance constrained optimal control problems in aerospace via
  kernel density estimation.
\newblock \emph{Optimal Control Applications and Methods}, 39\penalty0
  (5):\penalty0 1833--1858, 2018.

\bibitem[Cherukuri et~al.(2016)Cherukuri, Mallada, and
  Cort\'{e}s]{Cherukuri16a}
A.~Cherukuri, E.~Mallada, and J.~Cort\'{e}s.
\newblock Asymptotic convergence of constrained primal--dual dynamics.
\newblock \emph{Systems \& Control Letters}, 87:\penalty0 10--15, 2016.

\bibitem[Das and Dennis(1998)]{Das98n}
I.~Das and J.E. Dennis.
\newblock Normal-boundary intersection: {A} new method for generating the
  {P}areto surface in nonlinear multicriteria optimization problems.
\newblock \emph{SIAM Journal on Optimization}, 8[3]:\penalty0 631--657, 1998.

\bibitem[Ehrgott(2005)]{Ehrgott05m}
M.~Ehrgott.
\newblock \emph{Multicriteria Optimization}.
\newblock Springer, 2005.

\bibitem[Geibel and Wysotzki(2005)]{Geibel05r}
Peter Geibel and Fritz Wysotzki.
\newblock Risk-sensitive reinforcement learning applied to control under
  constraints.
\newblock \emph{Journal of Artificial Intelligence Research}, 24:\penalty0
  81--108, 2005.

\bibitem[Gharaibeh et~al.(2017)Gharaibeh, Salahuddin, Hussini, Khreishah,
  Khalil, Guizani, and Al-Fuqaha]{Gharaibeh17s}
Ammar Gharaibeh, Mohammad~A Salahuddin, Sayed~Jahed Hussini, Abdallah
  Khreishah, Issa Khalil, Mohsen Guizani, and Ala Al-Fuqaha.
\newblock Smart cities: A survey on data management, security, and enabling
  technologies.
\newblock \emph{IEEE Communications Surveys \& Tutorials}, 19\penalty0
  (4):\penalty0 2456--2501, 2017.

\bibitem[Howard and Matheson(1972)]{Howard72r}
Ronald~A Howard and James~E Matheson.
\newblock Risk-sensitive {M}arkov decision processes.
\newblock \emph{Management Science}, 18\penalty0 (7):\penalty0 356--369, 1972.

\bibitem[Johansson(1993)]{Johansson93s}
Rolf Johansson.
\newblock \emph{System modeling and identification}, volume~1.
\newblock Prentice Hall Englewood Cliffs, NJ, 1993.

\bibitem[Kober et~al.(2013)Kober, Bagnell, and Peters]{Kober13r}
Jens Kober, J~Andrew Bagnell, and Jan Peters.
\newblock Reinforcement learning in robotics: {A} survey.
\newblock \emph{The International Journal of Robotics Research}, 32\penalty0
  (11):\penalty0 1238--1274, 2013.

\bibitem[Kokotovi{\'c}(1991)]{Kokotovic91f}
Petar~V Kokotovi{\'c}.
\newblock \emph{Foundations of adaptive control}.
\newblock Springer, 1991.

\bibitem[Lewis(1974)]{Lewis74c}
D.~Lewis.
\newblock Causation.
\newblock \emph{The Journal of Philosophy}, 70\penalty0 (17):\penalty0
  556--567, 1974.

\bibitem[Li et~al.(2000)Li, Wendt, and Wozny]{Li00r}
P.~Li, M.~Wendt, and G.~Wozny.
\newblock Robust model predictive control under chance constraints.
\newblock \emph{Computers \& Chemical Engineering}, 24[2-7]:\penalty0 829--834,
  2000.

\bibitem[Messac et~al.(2003)Messac, {Ismail-Yahaya}, and Mattson]{Messac03t}
A.~Messac, A.~{Ismail-Yahaya}, and C.A. Mattson.
\newblock The normalized normal constraint method for generating the {P}areto
  frontier.
\newblock \emph{Structural and Multidisciplinary Optimization}, 25[2]:\penalty0
  86--98, 2003.

\bibitem[Miettinen(1998)]{Miettinen98n}
K.~Miettinen.
\newblock \emph{Nonlinear Multiobjective Optimization}.
\newblock Springer, 1998.

\bibitem[{Mueller-Gritschneder} et~al.(2009){Mueller-Gritschneder}, Graeb, and
  Schlichtmann]{Mueller-Gritschneder09a}
D.~{Mueller-Gritschneder}, H.~Graeb, and U.~Schlichtmann.
\newblock A successive approach to compute the bounded {P}areto front of
  practical multiobjective optimization problems.
\newblock \emph{SIAM Journal on Optimization}, 20[2]:\penalty0 915--934, 2009.

\bibitem[Nagurney and Zhang(1996)]{Nagurney96p}
A.~Nagurney and D.~Zhang.
\newblock \emph{Projected Dynamical Systems and Variational Inequalities with
  Applications}.
\newblock Springer, 1996.

\bibitem[Ono et~al.(2010)Ono, Blackmore, and Williams]{Ono10c}
Masahiro Ono, Lars Blackmore, and Brian~C Williams.
\newblock Chance constrained finite horizon optimal control with nonconvex
  constraints.
\newblock In \emph{Proceedings of the 2010 American Control Conference}, pages
  1145--1152. IEEE, 2010.

\bibitem[Paternain et~al.(2019{\natexlab{a}})Paternain, Calvo-Fullana, Chamon,
  and Ribeiro]{Paternain19s}
Santiago Paternain, Miguel Calvo-Fullana, Luiz~FO Chamon, and Alejandro
  Ribeiro.
\newblock Safe policies for reinforcement learning via primal-dual methods.
\newblock \emph{arXiv preprint arXiv:1911.09101}, 2019{\natexlab{a}}.

\bibitem[Paternain et~al.(2019{\natexlab{b}})Paternain, Chamon, Calvo-Fullana,
  and Ribeiro]{Paternain19c}
Santiago Paternain, Luiz~F.O. Chamon, Miguel Calvo-Fullana, and Alejandro
  Ribeiro.
\newblock Constrained reinforcement learning has zero duality gap.
\newblock In \emph{Advances in Neural Information Processing Systems}, pages
  7553--7563, 2019{\natexlab{b}}.

\bibitem[Pearl(2009)]{Pearl09c}
J.~Pearl.
\newblock \emph{Causality: {M}odels, Reasoning and Inference}.
\newblock Cambridge University Press, 2009.

\bibitem[Rawlings et~al.(2017)Rawlings, Mayne, and Diehl]{Rawlings17m}
J.B. Rawlings, D.Q. Mayne, and M.M. Diehl.
\newblock \emph{Model Predictive Control: {T}heory, Computation, and Design}.
\newblock Nob Hill Publishing, 2017.

\bibitem[Schwarm and Nikolaou(1999)]{Schwarm99c}
A.T. Schwarm and M.~Nikolaou.
\newblock Chance-constrained model predictive control.
\newblock \emph{AIChE Journal}, 45[8]:\penalty0 1743--1752, 1999.

\bibitem[Shapiro(2000)]{Shapiro00d}
A.~Shapiro.
\newblock Duality, optimality conditions, and perturbation analysis.
\newblock In H.~Wolkowicz, R.~Saigal, and L.~Vandenberghe, editors,
  \emph{Handbook of Semidefinite Programming: {T}heory, Algorithms, and
  Applications}, pages 67--91. Springer, 2000.

\bibitem[Sutton and Barto(2018)]{Sutton18r}
Richard~S Sutton and Andrew~G Barto.
\newblock \emph{Reinforcement learning: {A}n introduction}.
\newblock MIT press, 2018.

\bibitem[Woodward(2005)]{Woodward05m}
J.~Woodward.
\newblock \emph{Making things happen: {A} theory of causal explanation}.
\newblock Oxford University Press, 2005.

\end{thebibliography}

\end{document}